\documentclass[11pt]{article}

\usepackage[margin=1in]{geometry}
\usepackage[T1]{fontenc}
\usepackage[utf8]{inputenc}
\usepackage{amsmath,amssymb,amsthm,mathtools}
\usepackage{enumitem}
\usepackage{microtype}
\usepackage{hyperref}
\usepackage{mathrsfs}

\hypersetup{
  colorlinks=true,
  linkcolor=blue,
  citecolor=blue,
  urlcolor=blue
}

\newtheorem{theorem}{Theorem}[section]
\newtheorem{lemma}[theorem]{Lemma}
\newtheorem{proposition}[theorem]{Proposition}
\newtheorem{corollary}[theorem]{Corollary}
\theoremstyle{definition}
\newtheorem{definition}[theorem]{Definition}
\theoremstyle{remark}

\newtheorem{claim}[theorem]{Claim}

\newcommand{\E}{\mathbb E}
\newcommand{\R}{\mathbb R}
\newcommand{\Z}{\mathbb Z}
\newcommand{\F}{\mathbb F}
\newcommand{\1}{\mathbf 1}
\newcommand{\homn}{\operatorname{hom}}

\title{The sparse regularity method for graphs with few even-cycle homomorphisms}
\author{
Nan Zhang\thanks{Corresponding author. Center for Discrete Mathematics, Fuzhou University,
Fuzhou, 350108, P. R. China. Email: \texttt{1989378671@qq.com}.}
}
\date{}

\begin{document}
\maketitle

\begin{abstract}
Let $C_k$ denote a cycle on $k$ vertices. Pascadi (2026) established a removal lemma for
$C_{2k+1}$ in graphs with few $C_{2k}$-homomorphisms under a dense-pairs condition, using
Schatten--von Neumann norms and entropy methods. We give a short proof of a version with
more restrictive partition and diagonal requirements, extending the combinatorial approach
of Conlon, Fox, Sudakov and Zhao (2021) without using Pascadi's Schatten-norm regularity
framework. We illustrate the method with a known sparse removal statement for odd product
equations in finite groups.

\medskip
\noindent\textbf{Keywords:} Weak sparse regularity lemma; Removal lemma; Homomorphism density
\end{abstract}

\section{Introduction}\label{sec:introduction}

Szemer\'edi's regularity lemma \cite{SzemReg} was a key ingredient in his proof \cite{SzemAP}
of the famous conjecture of Erd\H{o}s and Tur\'an \cite{ET} that every sequence of integers with
positive upper density contains arbitrarily long arithmetic progressions. Roughly speaking, the
lemma states that the vertex set of any graph can be partitioned into a small number of parts such
that almost every pair of parts induces a pseudorandom (or random-looking) subgraph. It has
proven to be a powerful tool in extremal graph theory.

As a notable application of the regularity lemma, Ruzsa and Szemer\'edi \cite{RSz} established
the following well-known triangle removal lemma, which itself implies Roth's theorem \cite{Roth}
on 3-term arithmetic progressions.

\begin{theorem}[Ruzsa and Szemer\'edi \cite{RSz}]\label{thm:dense-triangle-removal}
For every $\epsilon>0$, there exists $\delta>0$ such that any $n$-vertex graph $G$ with fewer
than $\delta n^3$ triangles can be made triangle-free by removing at most $\epsilon n^2$ edges.
\end{theorem}

This concept generalizes to an $H$-removal lemma for any fixed subgraph $H$ in (weighted)
graphs (see \cite{EFR,FoxRemoval,ConlonFoxSurvey}) and to hypergraphs (see \cite{Gowers,Tao}),
with the latter leading to a proof of Szemer\'edi's theorem for $k$-term arithmetic progressions
\cite{SzemAP,TaoErgodic}.

Initially, the regularity lemma was applied primarily to dense graphs (see the survey
\cite{KomlosSimonovits} and related works). Given the prevalence of problems concerning sparse
graphs---where an $n$-vertex graph is sparse if it contains $o(n^2)$ edges---it is natural to
extend the regularity method to this setting. The sparse regularity lemma was first studied by
Kohayakawa and R\"odl \cite{Kohayakawa,KR} (see also \cite{GerkeSteger}), and later Scott
\cite{Scott} removed the need for an upper uniformity condition.

However, both the regularity lemma and the accompanying counting lemma behave quite differently
in sparse graphs compared to dense ones. To overcome these difficulties, one typically requires the
host sparse graph to be well-behaved---such as random or pseudorandom---to ensure a counting
lemma holds. For random graphs, the KLR conjecture \cite{KLR2}, a probabilistic embedding
statement, was proved using hypergraph containers \cite{BMS,SaxtonThomason}; a related counting
lemma for subgraphs of random graphs was established in \cite{KLR}. Related transference
results extend classical theorems such as Tur\'an's theorem and Szemer\'edi's theorem
to sparse random settings \cite{CG16,Schacht}. In the pseudorandom setting, the celebrated
Green--Tao theorem \cite{GreenTao} on arithmetic progressions in the primes is a prime example.
Conlon, Fox and Zhao \cite{CFZGT} later gave a simplified exposition; they also proved a relative
hypergraph removal lemma yielding a relative Szemer\'edi theorem \cite{CFZRelative}.

In analogy with Theorem~\ref{thm:dense-triangle-removal}, it is natural to investigate a removal
lemma for odd cycles $C_{2m+1}$ in sparse graphs with few copies of $C_{2m}$. Recently, motivated
in part by the study of quasi-random graphs \cite{CGW,Thomason1}, Conlon, Fox, Sudakov, and Zhao
\cite{CFSZ21,CFSZ22} developed a sparse regularity method without assuming that the graphs lie in
highly pseudorandom hosts. In particular, they proved the following result.

\begin{theorem}[Conlon, Fox, Sudakov and Zhao \cite{CFSZ21}]\label{thm:cfsz21}
An $n$-vertex graph with $o(n^2)$ copies of $C_4$ and $o(n^{5/2})$ copies of $C_5$ can be made
$\{C_3,C_5\}$-free by removing $o(n^{3/2})$ edges.
\end{theorem}

They also provided the following applications:
\begin{itemize}
\item Every $n$-vertex graph with no $5$-cycle can be made triangle-free by deleting
$o(n^{3/2})$ edges.
\item For $r\ge 3$, every $n$-vertex $r$-graph with girth greater than $5$ has
$o(n^{3/2})$ edges.
\item Every subset of $[n]$ without a nontrivial solution to
$x_1+x_2+2x_3=x_4+3x_5$ has size $o(\sqrt n)$.
\end{itemize}

Since a $C_4$-free graph on $n$ vertices has at most $O(n^{3/2})$ edges, a removal lemma in such
graphs is only meaningful if one can remove $o(n^{3/2})$ edges to achieve the desired property
(here, being $C_5$-free).

Theorem~\ref{thm:cfsz21} is essentially best possible. For $pn\to\infty$, the number of $C_m$'s
in $G(n,p)$ is typically of order $p^m n^m$ for each fixed $m$. Moreover, if
$p\gtrsim n^{-1/2}$, a second-moment calculation shows that $G(n,p)$ typically contains
$\Theta(pn^2)$ edge-disjoint triangles and therefore cannot be made triangle-free by removing
$o(pn^2)$ edges. Consequently, the random graph $G(n,p)$ with $p=n^{-1/2}$ demonstrates that
Theorem~\ref{thm:cfsz21} fails if we only assume $O(n^{5/2})$ copies of $C_5$ and $O(n^2)$
copies of $C_4$.

A homomorphism from a graph $F$ to a graph $G$ is a mapping
$\phi:V(F)\to V(G)$ such that $\phi(u)\phi(v)$ is an edge of $G$ whenever $uv$ is an edge of
$F$. Let $\homn(F,G)$ denote the number of homomorphisms from $F$ to $G$.

\begin{definition}[Homomorphism density]\label{def:hom-density-intro}
For two graphs $F$ and $G$, the $F$-density in $G$ is defined to be
\[
t(F,G)=\frac{\homn(F,G)}{|V(G)|^{|V(F)|}}.
\]
\end{definition}

Building on Theorem~\ref{thm:cfsz21}, Pascadi \cite{Pascadi} established a removal lemma for
longer odd cycles under mild pseudorandomness assumptions. His proof relies on a regularity
lemma and counting lemmas that employ Schatten--von Neumann norms to measure uniformity.
We consider the following version.

\begin{theorem}[A version of Pascadi's theorem \cite{Pascadi}]\label{thm:pascadi}
For every $\epsilon,L,C>0$ and integer $k\ge2$, there exist positive integers $n_0,M$ and $\delta>0$ such that the
following holds. If $n>n_0$, $p\ge0$, and $G$ is an $n$-vertex graph satisfying:
\begin{enumerate}[label=(\roman*)]
\item For every partition $V(G)=V_1\cup\cdots\cup V_t$ into at most $M$ parts, at most $\epsilon pn^2/4$ edges lie
between pairs $(V_i,V_j)$, including $i=j$, with
\[
e(V_i,V_j)\ge Lp|V_i||V_j|.
\]
Here $e(X,Y)=\sum_{x\in X,y\in Y}\1_G(x,y)$, so internal edges are counted twice in $e(X,X)$; edges lying in the union of the specified pairs are counted once.
\item
\[
t(C_{2k},G)\le (Cp)^{2k},
\qquad\text{i.e.}\qquad
\homn(C_{2k},G)\le (Cpn)^{2k}.
\]
\item $G$ contains at most $\delta p^{2k+1}n^{2k+1}$ copies of $C_{2k+1}$.
\end{enumerate}
Then $G$ can be made $C_3,C_5,\ldots,C_{2k+1}$-free by removing at most $\epsilon pn^2$
edges.
\end{theorem}

The dense-pairs hypothesis in Theorem~\ref{thm:pascadi} is more restrictive in its
partition and diagonal requirements than that in \cite[Theorem~9 and Remark~11]{Pascadi},
with a different normalization of constants: here we consider all partitions into at
most $M$ parts and include diagonal pairs, whereas Pascadi considers partitions into
parts of size at least $\tau n$ and only off-diagonal pairs.
In this paper, we give a short proof of the version stated above
(Section~\ref{sec:proofs}), extending the method of Conlon, Fox, Sudakov and Zhao \cite{CFSZ21}. In Section~\ref{sec:applications}, we illustrate the method with a sparse
removal statement for odd product equations in finite groups under a local $C_4$-energy
condition. This statement is also covered by \cite[Corollary~76]{Pascadi}.

We also note the following consequence of Pascadi's theorem. A direct proof using our methods
is given in Section~\ref{sec:proofs}.

\begin{theorem}\label{thm:C4-removal}
Let $p\ge n^{-1/2}$ and $k\ge2$ be an integer. An $n$-vertex graph with
$o(p^4n^4)$ copies of $C_4$ and $o(p^{2k+1}n^{2k+1})$ copies of $C_{2k+1}$ can be made
$C_3,C_5,\ldots,C_{2k+1}$-free by removing $o(pn^2)$ edges.
\end{theorem}

Setting $p=n^{-1/2}$ in Theorem~\ref{thm:C4-removal} yields a cleaner statement in terms of
cycle counts.

\begin{corollary}\label{cor:C2k1-removal}
Let $k\ge2$ be an integer. An $n$-vertex graph with $o(n^2)$ copies of $C_4$ and
$o(n^{k+1/2})$ copies of $C_{2k+1}$ can be made $C_3,C_5,\ldots,C_{2k+1}$-free by removing
$o(n^{3/2})$ edges.
\end{corollary}

The $C_4$ hypothesis in
Corollary~\ref{cor:C2k1-removal} cannot be omitted, as shown by the following construction
(presented in Section~\ref{sec:tight}).

\begin{proposition}\label{prop:tight}
For any fixed $k\ge2$ and all sufficiently large $n$, there exists an $n$-vertex graph with
$O\!\left(n^2(1+1/8)^{\log_3 n}\right)$ copies of $C_4$ and
$O\!\left(n^{k+1/2}(1-2^{-2k})^{\log_3 n}\right)$ copies of $C_{2k+1}$ that cannot be made triangle-free by deleting
$o(n^{3/2})$ edges.
\end{proposition}

\paragraph{Organization.}
In Section~\ref{sec:tight}, we prove Proposition~\ref{prop:tight}, showing that the $C_4$
hypothesis in Corollary~\ref{cor:C2k1-removal} cannot be omitted. Section~\ref{sec:prelim} introduces a weak
sparse regularity lemma and collects auxiliary results. Section~\ref{sec:counting} contains a
sparse counting lemma for odd cycles. Section~\ref{sec:proofs} provides the short proofs of
Theorem~\ref{thm:pascadi} and Theorem~\ref{thm:C4-removal}. In
Section~\ref{sec:applications}, we illustrate the method with sparse removal for odd product
equations in finite groups.

\section{A construction illustrating the $C_4$ condition}\label{sec:tight}

We extend the construction in the proof of \cite[Proposition~1.5]{CFSZ21} to longer odd cycles.

Let us restate Proposition~\ref{prop:tight} as follows:

\emph{For any fixed $k\ge2$, there exists an $n$-vertex graph with
$O\bigl(n^2(1+\frac18)^{\log_3 n}\bigr)$ copies of $C_4$ and
$O\bigl(n^{k+1/2}(1-\frac1{2^{2k}})^{\log_3 n}\bigr)$ copies of $C_{2k+1}$
that cannot be made triangle-free by deleting $o(n^{3/2})$ edges.}

The tensor product $G\otimes H$ of two graphs $G$ and $H$ is the graph with vertex set
$V(G)\times V(H)$, where $((x_1,y_1),(x_2,y_2))\in E(G\otimes H)$ if and only if
$(x_1,x_2)\in E(G)$ and $(y_1,y_2)\in E(H)$. Let $m$ be a positive integer, we may then
define $G^{\otimes m}$ inductively by $G^{\otimes1}=G$ and
$G^{\otimes m}=G\otimes G^{\otimes(m-1)}$. A key property, which we will need below, is that
$t(H,G^{\otimes m})=t(H,G)^m$ for any graph $H$ and $G$ and any positive integer $m$.

\begin{proof}[\normalfont\bfseries Proof of Proposition~\ref{prop:tight}]
Let $m$ be a positive integer and $G=C_3^{\otimes m}$. In other words, its vertices can be
labeled by $\F_3^m$ and two vertices are adjacent if and only if they differ in every coordinate.
Note that every edge of $G$ lies in a unique triangle. For three vertices on $C_3$, to make each
coordinate different, there are $6$ choices for each coordinate, thus this graph has $6^m/6$
triangles. We write $\homn(H,G)$ for the number of homomorphisms from $H$ to $G$.
The number of the closed walks of length $k$ in a triangle $C_3$ is equal to the $k$th moment
of the eigenvalues of the adjacency matrix of $C_3$ and so
$\homn(C_k,C_3)=2^k+2(-1)^k$. By the property of the tensor product, we have
$\homn(C_k,G)=\homn(C_k,C_3)^m=(2^k+2(-1)^k)^m$.

Let $G'$ be the graph obtained from $G$ by keeping every triangle independently with probability
$p=(\frac{\sqrt3}{2})^m$, so that in expectation $p6^{m-1}=3^{3m/2}/6$ triangles remain in $G'$.
Let $k\ge4$ be an integer. We shall estimate the expected number of $C_k$ in $G'$.

Fix $i$ triangles in $G$, $\lceil k/2\rceil\le i\le k$, let $\lambda_i$ be the number of
$C_k$ that exactly intersect these $i$ triangles in edges. Then there are $k-i$ triangles that
intersect this $C_k$ with two consecutive edges in $G$, and so we have
\[
\lambda_i\le\binom{i}{k-i}3^i\le2^{O(k)}.
\]
Conversely, note that such a class of $i$ triangles would yield a cycle $C_i$ of length $i$,
by replacing each two-edge segment by its third edge. The cases $i=2,3$ would put an edge
in two different triangles, so $i\ge4$; each resulting $C_i$ determines its triangles.
Together with $\homn(C_i,G)=(2^i+2(-1)^i)^m$ we obtain that the expected number of $C_k$
in $G'$ is at most
\[
\sum_{i=\max\{4,\lceil k/2\rceil\}}^k
\binom{i}{k-i}3^i p^i(2^i+2(-1)^i)^m
\le2^{O(k)}
\sum_{i=\max\{4,\lceil k/2\rceil\}}^k p^i(2^i+2(-1)^i)^m.
\]
Consecutive summands for $i\ge4$ increase by a factor at least $(5\sqrt3/6)^m>1$,
so this is at most $k2^{O(k)}(3^{k/2}(1-\frac1{2^{k-1}}))^m$ if $k$ is odd, and at most
$k2^{O(k)}(3^{k/2}(1+\frac1{2^{k-1}}))^m$ if $k$ is even.
Let $n=3^m$ and return to the fixed $k\ge2$ in the proposition.
Then the expected number of $C_{2k+1}$ in $G'$ is at most
\[
(2k+1)2^{O(2k+1)}
\left(3^{k+1/2}\left(1-\frac1{2^{2k}}\right)\right)^m
=O\!\left(n^{k+1/2}\left(1-\frac1{2^{2k}}\right)^{\log_3 n}\right)
\]
and the expected number of $C_4$ in $G'$ is $O(n^2(1+\frac18)^{\log_3 n})$.

The retained triangle count is binomial with mean $n^{3/2}/6$, so it is $\Theta(n^{3/2})$
with probability $1-o(1)$ by Chebyshev's inequality.
Therefore, by Markov's inequality, with positive probability, $G'$ is a graph on $n$ vertices
which is an edge-disjoint union of $\Theta(n^{3/2})$ triangles with
$O(n^{k+1/2}(1-\frac1{2^{2k}})^{\log_3 n})$ $C_{2k+1}$'s and
$O(n^2(1+\frac18)^{\log_3 n})$ $C_4$'s.
For arbitrary sufficiently large $n$, take $m=\lfloor\log_3n\rfloor$ and add isolated vertices;
all bounds are preserved since $3^m=\Theta(n)$.
\end{proof}

\section{An illustrative application}\label{sec:applications}

We illustrate Theorem~\ref{thm:pascadi} by a group removal statement
also covered by \cite[Corollary~76]{Pascadi}.
For $A\subseteq\Gamma$, put
$q_A(t)=\#\{(a,b)\in A^2:ab^{-1}=t\}$ and
$Q(A)=\sum_{t\ne1_\Gamma}\binom{q_A(t)}2$.
Thus $Q(A)$ counts pairs of distinct representations of nonidentity quotients.

\begin{theorem}[Sparse removal for odd product equations]\label{thm:sparse-group-removal}
Let $k\ge2$ be fixed and put $m=2k+1$. For every $\epsilon>0$ there exist
$\delta>0$ and $N_0$ such that the following holds. Let $\Gamma$ be a group
of order $N\ge N_0$, let $A_1,\ldots,A_m\subseteq\Gamma$, and let
$g\in\Gamma$. Suppose that $\sum_iQ(A_i)\le\delta N$ and that
$x_1\cdots x_m=g$, with $x_i\in A_i$, has at most
$\delta N^{k-1/2}$ solutions. Then there are subsets $A_i'\subseteq A_i$
with $|A_i\setminus A_i'|\le\epsilon N^{1/2}$ for every $i$, such that the
equation has no solution with $x_i\in A_i'$ for all $i$.
\end{theorem}

\begin{proof}
Since right multiplication preserves $Q$, we may assume $g=1_\Gamma$.
Following Kr\'al, Serra and Vena \cite{KSV}, let $\mathcal G$ have vertex set
$\Gamma\times\{1,\ldots,m\}$, with layer indices taken modulo $m$.
Join $(y,i)$ to $(yx_i,i+1)$ for each $x_i\in A_i$, and label this edge by $[x_i,i]$.

Observe that any $C_m$ visits the layers in cyclic order: its $m$ layer steps are
$\pm1$ and sum to an odd multiple of $m$, hence to $\pm m$.
Its labels therefore give a solution to $x_1\cdots x_m=1_\Gamma$.
In the opposite way, each solution gives $N$ vertex-disjoint cycles
\[
(z,1),(zx_1,2),\ldots,(zx_1\cdots x_{m-1},m),(z,1),
\qquad z\in\Gamma,
\]
since any vertex determines $z$.
Thus $\mathcal G$ contains at most $\delta N^{m/2}$ copies of $C_m$.

Put $n=mN$, $p=N^{-1/2}$, and assume $\delta\le1$.
For a set $A$ of size $s$, we have $s(s-1)\le N-1+Q(A)$ and
$\sum_tq_A(t)^2=2s^2-s+2Q(A)$, so $\sum_tq_{A_i}(t)^2=O(N)$.
Let $H_i$ contain only the edges between layers $i,i+1$, keeping all $n$ vertices.
Two distinct representations of a nonidentity quotient give two common neighbours;
counting each four-cycle twice yields $N(C_4,H_i)=NQ(A_i)/2\le\delta N^2/2$.
Every other $C_4$ uses three consecutive layers and corresponds to $ab=cd$,
where $a,c\in A_i$ and $b,d\in A_{i+1}$.
The left and right quotient counts of each set have equal squared sums, since
$ab^{-1}=cd^{-1}\iff c^{-1}a=d^{-1}b$.
Hence Cauchy--Schwarz applied to $c^{-1}a=db^{-1}$ gives $O(N)$ such quadruples
and, allowing $N$ starting vertices, $O(N^2)$ such cycles.
Thus $N(C_4,\mathcal G)=O(N^2)$; as $p\ge n^{-1/2}$,
Lemma~\ref{lem:hom-C2k} gives $\homn(C_{2k},\mathcal G)\le(Cpn)^{2k}$,
with $C$ depending only on $m$.

Consider any partition into at most $M$ parts.
If $e_{\mathcal G}(U,W)\ge Lp|U||W|$, the $H_i$ with
$e_{H_i}(U,W)<Lp|U||W|/(2m)$ together contribute at most half these edges.
Applying Lemma~\ref{lem:dense-pair} to the remaining contributions and summing over $i$,
the number of edges in dense pairs is at most
\[
K_m\left(\frac{N^{3/2}}L+M^2N+
M^{1/4}\delta^{1/4}N^{3/2}\right).
\]
This includes diagonal pairs, counting internal edges twice in $e(U,U)$
and each edge once in the final bound.
Choose $L$ sufficiently large, and take $M,\delta_0,n_0$ from
Theorem~\ref{thm:pascadi} with deletion tolerance $\epsilon/m^3$.
Next choose $\delta\le\min\{1,\delta_0m^m\}$ sufficiently small and then
$N_0$ sufficiently large that $mN>n_0$ and the displayed bound is at most
$\epsilon N^{3/2}/(4m)$.
Since also $\delta N^{m/2}\le\delta_0p^mn^m$, Theorem~\ref{thm:pascadi}
makes $\mathcal G$ $C_m$-free by deleting at most $\epsilon N^{3/2}/m$ edges.

In $A_i$, remove $x_i$ if at least $N/m$ edges labelled $[x_i,i]$ have been deleted.
This removes at most $\epsilon\sqrt N$ elements from each $A_i$.
Let $A_i'$ denote the remaining set.
If a solution remained, each of its $N$ edge-disjoint cycles would contain a deleted edge.
By the pigeonhole principle, at least $N/m$ deleted edges would have one of its $m$ labels,
contradicting the definition of $A_i'$.
Multiplying the last remaining set by $g$ returns to the original equation.
The proof is complete.
\end{proof}

A set $S\subseteq[n]$ is \emph{Sidon} if every solution of
$x_1+x_2=x_3+x_4$ in $S$ has $\{x_1,x_2\}=\{x_3,x_4\}$ as multisets.
The following consequence is known: its five-variable case is
\cite[Theorem~1.14]{CFSZ21}, the extension to longer equations is noted
in the footnote on p.~2380 there, and Prendiville
\cite[Corollary~1.3]{Prendiville} gives the stronger
bound $O_{b_1,\ldots,b_{2k}}(\sqrt n/\log\log n)$ for more general equations.

\begin{samepage}
\begin{corollary}[A Sidon-set consequence]\label{cor:sidon-average}
Fix $k\ge2$ and positive integers $b_1,\ldots,b_{2k}$, and put
$B=b_1+\cdots+b_{2k}$. If a Sidon set $S\subseteq[n]$ contains no solution
with pairwise distinct variables to
\[
 b_1x_1+\cdots+b_{2k}x_{2k}=Bx_{2k+1},
\]
then $|S|=o(n^{1/2})$.
\end{corollary}
\end{samepage}

\begin{proof}
Assume $S\ne\varnothing$, set $s=|S|$, and write
$f(t)=\sum_{x\in S}e^{2\pi ixt}$. Unique representation of nonzero
differences gives $s(s-1)\le2(n-1)$, hence $s=O(\sqrt n)$.
For nonzero integers $c_1,\ldots,c_{2k}$, Fourier orthogonality,
H\"older's inequality and periodicity give
\[
\begin{aligned}
 \#\{\mathbf x\in S^{2k}:\textstyle\sum_i c_ix_i=0\}
 &\le\int_0^1|f(t)|^{2k}\,dt
 \le s^{2k-4}\int_0^1|f(t)|^4\,dt\\
 &=s^{2k-4}(2s^2-s)=O(n^{k-1}),
\end{aligned}
\]
using $|f|\le s$ and the Sidon fourth-moment identity.
Every solution in the corollary has two equal variables; merging them
leaves nonzero coefficients, since the merged coefficient is either
$b_i+b_j>0$ or $b_i-B<0$. Summing over all pairs bounds the total
number of solutions by $O(n^{k-1})=o(n^{k-1/2})$.

Put $(a_1,\ldots,a_{2k+1})=(b_1,\ldots,b_{2k},-B)$.
Choose $N>2Bn$ with $N\equiv1\pmod{B\prod_i b_i}$ and $N=O(n)$,
possible because the coefficients are fixed.
Modulo $N$, neither the difference equalities nor $\sum_i a_ix_i=0$
acquires additional solutions, and each $a_i$ is invertible.
Thus $A_i=a_iS\subseteq\Z/N\Z$ satisfies $Q(A_i)=0$.
As $N=\Theta(n)$, Theorem~\ref{thm:sparse-group-removal}, in additive
notation, eliminates all solutions by deleting $o(\sqrt n)$ elements
from each $A_i$. Pulling these deletions back under multiplication by
$a_i$ gives sets $D_i\subseteq S$, each of size $o(\sqrt n)$.
Any $x\in S\setminus\bigcup_iD_i$ would still give the diagonal solution,
since $\sum_i a_i=0$. Hence $S=\bigcup_iD_i$ and
$|S|\le\sum_i|D_i|=o(\sqrt n)$.
\end{proof}

\section{Useful lemmas}\label{sec:prelim}

\subsection{A weak sparse regularity lemma}\label{subsec:weak-reg}

Given a pair of probability spaces $V_1$ and $V_2$, usually vertex sets with uniform measure,
we define the cut norm for a measurable function $f:V_1\times V_2\to\R$ by
\begin{equation}\label{eq:cutnorm}
\|f\|_\square
\coloneqq
\sup_{A\subseteq V_1,\ B\subseteq V_2}
\left|
\E_{x\in V_1,y\in V_2}
f(x,y)\1_A(x)\1_B(y)
\right|.
\end{equation}
Here $A$ and $B$ range over all measurable subsets and $x,y$ are chosen independently according
to the corresponding probability measures.

Given a partition $\mathcal P$ of a probability space $V$ and a function
$f:V\times V\to\R$, we write
\[
f_{\mathcal P}:V\times V\to\R
\]
for the function obtained from $f$ by averaging over blocks $A\times B$ where $A$ and $B$ are
parts of $\mathcal P$; that is,
\[
f_{\mathcal P}(x,y)
=
\frac{1}{\mu(A)\mu(B)}
\int_{A\times B} f
\qquad\text{for }(x,y)\in A\times B,
\]
where $\mu(\cdot)$ is the probability measure on $V$ and zero-measure parts are ignored.

The weak regularity lemma for kernels of Frieze and Kannan \cite{FK99}
(see also \cite[Corollary 9.13]{Lovasz}) is a related statement about approximation by
stepfunctions in the cut norm. Furthermore, the step function can be obtained by averaging the
original function over steps; such steppings are almost optimal. This analytic perspective has
been popularized by the development of graph limits \cite{LovaszSzegedy}.

Let $(f-f_{\mathcal P})\1_{f_{\mathcal P}\le1}$ denote the function with value
$f(x,y)-f_{\mathcal P}(x,y)$ if $f_{\mathcal P}(x,y)\le1$ and $0$ otherwise. We will apply the
following weak sparse regularity lemma \cite{CFSZ21}, in which we neglect those parts of the graph
that are too dense through the cutoff term $\1_{f_{\mathcal P}\le1}$.

\begin{theorem}[Conlon, Fox, Sudakov and Zhao \cite{CFSZ21}]\label{thm:weak-sparse-reg}
Let $\epsilon>0$, and let $V$ be a probability space and
$f:V\times V\to[0,\infty)$ be a measurable symmetric function. Then there exists a partition
$\mathcal P$ of $V$ into at most
\[
2^{32\E(f)\epsilon^{-2}}
\]
parts such that
\[
\|(f-f_{\mathcal P})\1_{f_{\mathcal P}\le1}\|_\square\le\epsilon.
\]
\end{theorem}

\subsection{Homomorphism density}\label{subsec:hom-density}

Recall the graph homomorphism density $t(H,G)$ from Definition~\ref{def:hom-density-intro}.

Let $V$ be a probability space and let $f:V\times V\to[0,\infty)$ be measurable and symmetric.

\begin{definition}[Homomorphism density in function]\label{def:hom-density-function}
Let $H$ be a graph. The $H$-density in $f$ is
\[
t(H,f)
=
\int_{V^{V(H)}}
\prod_{ij\in E(H)} f(x_i,x_j)
\prod_{i\in V(H)}dx_i.
\]
\end{definition}

We should think of $f$ as a normalized edge-indicator function of the form
$p^{-1}\1_{E(G)}$.

\begin{lemma}\label{lem:hom-C4}
Let $C>0$, $p\ge n^{-1/2}$, and suppose that $G$ has at most $(Cpn)^4$ copies of $C_4$.
Then
\[
\homn(C_4,G)\le(24C^4+3)p^4n^4.
\]
\end{lemma}

\begin{proof}
Let $T$ be the number of unlabelled copies of $C_4$ in $G$, write $d(v)$ for the degree of $v$, and let $r(x,y)$
denote the number of walks of length $2$ from $x$ to $y$. In particular,
$r(x,x)=d(x)$. Counting closed walks gives
\[
 \homn(C_4,G)=8T+2\sum_v d(v)^2-2e(G).
\]
For every nonnegative integer $r$, $r\le2\binom r2+1$. Since
$\sum_{x\ne y}\binom{r(x,y)}2=4T$, we have
\[
 \sum_vd(v)^2=\sum_{x\ne y}r(x,y)+2e(G)
 \le8T+n(n-1)+2e(G).
\]
Consequently, $\homn(C_4,G)\le24T+3n^2$. As $p\ge n^{-1/2}$,
this is at most $(24C^4+3)p^4n^4$, proving the claim.
\end{proof}

\begin{lemma}[Galvin and Tetali \cite{GalvinTetali}]\label{lem:galvin-tetali}
For every $h$-vertex $d$-regular bipartite graph $H$ with $d\ge1$ and any graph $G$,
\[
\homn(H,G)\le \homn(K_{d,d},G)^{h/(2d)}.
\]
\end{lemma}

\begin{lemma}\label{lem:hom-C2k}
Let $k\ge2$ be an integer and let $p\ge n^{-1/2}$. If $G$ contains at most $O(p^4n^4)$ copies of $C_4$, then
\[
\homn(C_{2k},G)=O(p^{2k}n^{2k}).
\]
\end{lemma}

\begin{proof}
By Lemma~\ref{lem:hom-C4},
\[
\homn(C_4,G)=O(p^4n^4).
\]
Thus, by Lemma~\ref{lem:galvin-tetali},
\[
\homn(C_{2k},G)
\le
\homn(C_4,G)^{2k/4}
=
O(p^{2k}n^{2k}).
\]
\end{proof}

\begin{lemma}[Conlon, Fox, Sudakov and Zhao \cite{CFSZ21}]\label{lem:dense-pair}
Let $G$ be an $n$-vertex graph, and let
\[
V(G)=\bigcup_{i=1}^M V_i
\]
be a partition of the vertex set. Let $q>0$ and let $T$ denote the number of $4$-cycles in $G$.
Then the number of edges of $G$ lying between parts $V_i$ and $V_j$ with
\[
e(V_i,V_j)\ge q|V_i||V_j|
\]
(where $i=j$ is allowed) is
\[
O\!\left(\left(\frac1q+M^2+M^{1/4}T^{1/4}\right)n\right).
\]
\end{lemma}

\section{A sparse counting lemma for $C_{2k+1}$}\label{sec:counting}

Let $V$ be a probability space. Given a symmetric function
$f:V\times V\to[0,\infty)$, recall that the homomorphism density of $H$ in $f$ is
\[
t(H,f)
=
\int_{V^{V(H)}}\prod_{uv\in E(H)}f(x_u,x_v)\prod_{v\in V(H)}dx_v.
\]

Our counting lemma is as follows, extending Conlon, Fox, Sudakov and Zhao
\cite[Theorem 3.1]{CFSZ21}.

\begin{lemma}\label{lem:counting}
Let $k\in\Z$, $k\ge2$, $0<\epsilon<1$ and $C\ge1$. Let $V$ be a probability space and let
$f:V\times V\to[0,\infty)$ and $g:V\times V\to[0,1]$ be measurable symmetric functions
satisfying
\[
t(C_{2k},f)\le C
\qquad\text{and}\qquad
\|f-g\|_\square\le \epsilon^2.
\]
Then
\[
t(C_{2k+1},f)
\ge
t(C_{2k+1},g)-9Ck\,\epsilon^{1/2^{k+1}}.
\]
\end{lemma}

We will prove a multipartite version. Let
$V_1,\ldots,V_{2k+1}$ be probability spaces, with indices considered modulo $2k+1$.
For all distinct $i,j\in\Z/(2k+1)$, write $f_{ij}=f_{i,j}$ for a measurable function on
$V_i\times V_j$ and assume $f_{ij}(x_i,x_j)=f_{ji}(x_j,x_i)$.
The same transposition convention is used for $g_{ij}$.

Given $f_{ij}:V_i\times V_j\to\R$, write
$f_{12}\circ f_{23}\circ\cdots\circ f_{k,k+1}:V_1\times V_{k+1}\to\R$ for
\[
(f_{12}\circ\cdots\circ f_{k,k+1})(x_1,x_{k+1})
=
\E_{x_i\in V_i,\ i=2,\ldots,k}
\prod_{i=1}^k f_{i,i+1}(x_i,x_{i+1}).
\]

\begin{lemma}\label{lem:counting-multipartite}
Let $k\in\Z$, $k\ge2$, $0<\epsilon<1$ and $C\ge1$. Let
$V_1,\ldots,V_{2k+1}$ be probability spaces. For each $i$ modulo $2k+1$, let
\[
f_{i,i+1}:V_i\times V_{i+1}\to[0,\infty),
\qquad
g_{i,i+1}:V_i\times V_{i+1}\to[0,1]
\]
satisfy
\[
\|f_{i,i+1}-g_{i,i+1}\|_\square\le \epsilon^2,
\]
and
\[
\|f_{i,i+1}\circ\cdots\circ f_{i+(k-1),i+k}\|_2^2\le C.
\]
Then
\[
\E_{x_1,\ldots,x_{2k+1}}
\prod_{i=1}^{2k+1}f_{i,i+1}(x_i,x_{i+1})
\ge
\E_{x_1,\ldots,x_{2k+1}}
\prod_{i=1}^{2k+1}g_{i,i+1}(x_i,x_{i+1})
-
9Ck\,\epsilon^{1/2^{k+1}}.
\]
\end{lemma}

Note that the cut norm \eqref{eq:cutnorm} can be rephrased as
\begin{equation}\label{eq:cutnorm-equivalent}
\|f\|_\square
=
\sup_{\lambda_i:V_i\to[0,1],\,i=1,2}
\left|
\E_{x\in V_1,y\in V_2}
f(x,y)\lambda_1(x)\lambda_2(y)
\right|.
\end{equation}

We first record two basic lemmas from Conlon, Fox, Sudakov and Zhao \cite{CFSZ21}.

\begin{lemma}\label{lem:bounded-cut}
Let $f_{12}:V_1\times V_2\to\R$ and $f_{23}:V_2\times V_3\to[0,\infty)$ with
\[
\E_{x_3\in V_3} f_{23}(x_2,x_3)\le M
\]
for all $x_2\in V_2$. Then
\[
\|f_{12}\circ f_{23}\|_\square\le M\|f_{12}\|_\square.
\]
\end{lemma}

\begin{lemma}[{\cite[Lemma 3.4]{CFSZ21}}]\label{lem:triangle-counting}
Let $f_{ij}:V_i\times V_j\to\R$ for $ij\in\{12,13,23\}$ and assume that $f_{12}$ and
$f_{13}$ take only nonnegative values. Then
\[
\left|\E_{x_i\in V_i,\ i=1,2,3}
f_{12}(x_1,x_2)f_{13}(x_1,x_3)f_{23}(x_2,x_3)\right|
\le
\|f_{12}\|_\infty\|f_{13}\|_\infty\|f_{23}\|_\square.
\]
\end{lemma}

\begin{proof}[Proof of Lemma~\ref{lem:counting-multipartite}]
We first consider any oriented $k$-edge path, and number its edges
locally by $j=1,\ldots,k$. Write $u_j,v_j$ for the $f$- and
$g$-kernels on its $j$th edge, respectively, with the arguments
ordered along the path. Thus $u_j\ge0$, $0\le v_j\le1$, and
$\lVert u_j-v_j\rVert_\square\le\epsilon^2$.
Put $a_j=2^{-j}$, and define
\[
 S_j=\{x:\mathbb E_yu_j(x,y)\le\epsilon^{-a_j}\},
 \qquad \widehat u_j(x,y)=u_j(x,y)1_{S_j}(x).
\]
Since $\mathbb E u_j\le\mathbb E v_j+\epsilon^2\le2$,
the complement of $S_j$ has measure at most $2\epsilon^{a_j}$.
Therefore
\[
 \lVert\widehat u_j-v_j\rVert_\square
 \le\epsilon^2+2\epsilon^{a_j}\le3\epsilon^{a_j},
 \qquad
 \sup_x\mathbb E_y\widehat u_j(x,y)\le\epsilon^{-a_j}.
\]
For $1\le j\le k$, let
\[
 D_j=\lVert\widehat u_1\circ\cdots\circ\widehat u_j
                -v_1\circ\cdots\circ v_j\rVert_\square.
\]
Lemma~\ref{lem:bounded-cut}, also applied after transposing the kernels, gives
\[
 D_1\le3\epsilon^{a_1},\qquad
 D_j\le\epsilon^{-a_j}D_{j-1}+3\epsilon^{a_j}
       \le3j\epsilon^{a_j}.
\]
Here the last inequality follows inductively from
$a_{j-1}=2a_j$; the transposed application is valid because
every convolution of the $v_j$ takes values in $[0,1]$.

Apply this construction separately to the two paths
\[
 1,2,\ldots,k+1
 \qquad\text{and}\qquad
 1,2k+1,2k,\ldots,k+2.
\]
Let $H,J$ be the resulting convolutions of the clipped $f$-kernels,
and let $H_g,J_g$ be the corresponding convolutions of the
$g$-kernels. Thus
\[
 \lVert H-H_g\rVert_\square,
 \lVert J-J_g\rVert_\square\le3k\epsilon^{2^{-k}}.
\]
The clipped kernels are pointwise bounded by the original kernels.
The path hypotheses, with transposition for the second path,
therefore imply
\[
 \lVert H\rVert_2^2,\ \lVert J\rVert_2^2\le C.
\]
For a nonnegative function $h$ with $\lVert h\rVert_2^2\le C$
and any $R>0$, write $h_{\le R}=h1_{h\le R}$. Then
\[
 \lVert h-h_{\le R}\rVert_\square
 \le\mathbb E[h1_{h>R}]\le C/R.
\]
Set
\[
 a=2^{-(k+1)},\quad b=2^{-k}=2a,\quad
 A=\epsilon^{-a},\quad B=\epsilon^{-b}.
\]
Using $C\ge1$ and $k\ge2$, we obtain
\[
 \lVert H_{\le A}-H_g\rVert_\square\le4kC\epsilon^a,
 \qquad
 \lVert J_{\le B}-J_g\rVert_\square\le4kC\epsilon^b.
\]
All functions in the original cycle product are nonnegative.
Consequently its expectation is at least
\[
 \mathbb E_{x_1,x_{k+1},x_{k+2}}
 H_{\le A}(x_1,x_{k+1})J_{\le B}(x_1,x_{k+2})
 f_{k+1,k+2}(x_{k+1},x_{k+2}).
\]
In this expression replace, successively,
$f_{k+1,k+2}$ by $g_{k+1,k+2}$,
$J_{\le B}$ by $J_g$, and $H_{\le A}$ by $H_g$.
By Lemma~\ref{lem:triangle-counting} the sum of the absolute errors is at most
\[
 AB\epsilon^2+A\cdot4kC\epsilon^b+4kC\epsilon^a
 \le(8kC+1)\epsilon^a\le9kC\epsilon^a.
\]
Indeed, $AB\epsilon^2=\epsilon^{2-3a}\le\epsilon^a$.
The final expression is the $g$-cycle density, proving the lemma.
\end{proof}

\begin{proof}[Proof of Lemma~\ref{lem:counting}]
Apply Lemma~\ref{lem:counting-multipartite} with every probability space
equal to $V$ and every edge kernel equal to $f$ or $g$, respectively.
Symmetry gives $\|f\circ\cdots\circ f\|_2^2=t(C_{2k},f)$ for the
$k$-edge convolution, so all its hypotheses hold.
\end{proof}

\section{Proofs of Theorem~\ref{thm:pascadi} and Theorem~\ref{thm:C4-removal}}\label{sec:proofs}

We use the following graph removal lemma \cite{CFSZ21}, which allows both vertex and edge weights.

\begin{theorem}[Weighted graph removal lemma, dense setting]\label{thm:weighted-removal}
For every graph $F$ and $\epsilon>0$, there exists $\delta>0$ such that for a probability space
$V$ and measurable symmetric $g:V\times V\to[0,1]$ with $t(F,g)\le\delta$, there exists a
measurable symmetric subset $A\subseteq V\times V$ such that
\[
t(F,g\1_A)=0
\qquad\text{and}\qquad
\|g-g\1_A\|_1\le\epsilon.
\]
\end{theorem}

\begin{proof}[Proof of Theorem~\ref{thm:pascadi}]
If $p=0$, hypothesis (ii) forces $G$ to be empty, so the conclusion is
immediate. We may therefore assume $p>0$. Put $m=2k+1$ and
$T=m(m+1)$. Let $0<\delta<1$ be sufficiently small in terms of
$\epsilon,L,C,k$, to be chosen later. Let $V=V(G)$ with its original
uniform probability measure, and let
\[
f=p^{-1}\1_G:V\times V\to[0,p^{-1}]
\]
be the normalized edge-indicator function of $G$. From hypothesis (i), $G$
has at most $(L+\epsilon)pn^2/2$ edges, and hence $\E f\le L+\epsilon$.

Apply Theorem~\ref{thm:weak-sparse-reg} to $f/L$ to obtain a partition
$\mathcal P$ of $V$ into at most $M=M(k,L,\epsilon,\delta)$ parts such that
\[
 \|(f-f_{\mathcal P})\1_{f_{\mathcal P}\le L}\|_\square
 \le L\delta^{2^{k+2}}.
\]
Let
\[
 \widetilde f=f\1_{f_{\mathcal P}\le L},\qquad
 \widetilde g=\widetilde f_{\mathcal P}
             =f_{\mathcal P}\1_{f_{\mathcal P}\le L}.
\]
Then
\[
 \|\widetilde f-\widetilde g\|_\square\le L\delta^{2^{k+2}}.
\]
By hypothesis (i), forming $\widetilde f$ deletes at most
$\epsilon pn^2/4$ edges between dense pairs.

Let $S$ be the union of the parts of $\mathcal P$ of size at most $T$,
and put $U=V\setminus S$. Delete all remaining edges incident to $S$.
Since every remaining block has edge density at most $Lp$, this deletes
at most
\[
 Lp|S|n\le TMLpn\le\epsilon pn^2/3
\]
edges, provided $n>n_0$ is sufficiently large. Keep all vertices of $S$
as isolated vertices in the original probability space $V$, and replace
$\widetilde f,\widetilde g$ by
$\widetilde f\1_{U\times U},\widetilde g\1_{U\times U}$, respectively.
Because $S$ is a union of whole parts, these new functions still satisfy
\[
 \widetilde g=\widetilde f_{\mathcal P},\qquad
 \|\widetilde f-\widetilde g\|_\square\le L\delta^{2^{k+2}}.
\]

Partition each part $V_i\subseteq U$ into $m$ parts of nearly equal size,
and partition each part contained in $S$ arbitrarily into $m$ parts,
allowing empty parts:
\[
 V_i=V_i^{(1)}\cup\cdots\cup V_i^{(m)}.
\]
For $V_i\subseteq U$ we have $|V_i|>m(m+1)$, and hence
\[
 |V_i^{(r)}|\ge\lfloor |V_i|/m\rfloor\ge |V_i|/(m+1)
 \qquad(r\in[m]).
\]
Let $V^{(r)}=\bigcup_{V_i\in\mathcal P}V_i^{(r)}$. For $r,s\in[m]$, write
\begin{equation}\label{eq:frs-grs}
 f_{rs}=L^{-1}\widetilde f\,\1_{V^{(r)}\times V^{(s)}},\qquad
 g_{rs}=L^{-1}\widetilde g\,\1_{V^{(r)}\times V^{(s)}}.
\end{equation}
All these functions use the original uniform probability space $V$ in
each variable. Restricting a kernel to a rectangle does not increase its
cut norm, so
\[
 \|f_{rs}-g_{rs}\|_\square\le\delta^{2^{k+2}}.
\]
By hypothesis (ii),
\[
 t(C_{2k},f)\le C^{2k}.
\]

\begin{claim}\label{claim:L2-path}
We have
\[
 \|f_{1,2}\circ\cdots\circ f_{k,k+1}\|_2^2
 \le L^{-2k}t(C_{2k},f).
\]
\end{claim}

\begin{proof}
By definition,
\[
 (f_{1,2}\circ\cdots\circ f_{k,k+1})(x_1,x_{k+1})
 =\E_{x_2,\ldots,x_k\in V}
       \prod_{i=1}^kf_{i,i+1}(x_i,x_{i+1}).
\]
This is the number of $k$-edge paths between the endpoints, appropriately
normalized. Squaring and introducing independent copies of the internal
variables gives a $C_{2k}$-density. Discarding the nonnegative indicator
factors therefore gives the asserted inequality.
\end{proof}

The same argument applies to every cyclic $k$-edge arc. Thus all the path
$L^2$ squares in Lemma~\ref{lem:counting-multipartite} are at most
\[
 D=\max\{1,(C/L)^{2k}\}.
\]
Apply that lemma with $\epsilon_{\rm count}=\delta^{2^{k+1}}$, whose square
is $\delta^{2^{k+2}}$. Its error is $9kD\delta$. From
\eqref{eq:frs-grs} we obtain
\begin{equation}\label{eq:count-compare}
\begin{aligned}
 &\E_{x_1,\ldots,x_m\in V}
    \prod_{r=1}^m\widetilde f(x_r,x_{r+1})\1_{V^{(r)}}(x_r)\\
 &\quad\ge
  \E_{x_1,\ldots,x_m\in V}
    \prod_{r=1}^m\widetilde g(x_r,x_{r+1})\1_{V^{(r)}}(x_r)
       -9kDL^m\delta\\
 &\quad\ge (m+1)^{-m}t(C_m,\widetilde g)-9kDL^m\delta.
\end{aligned}
\end{equation}
Here $x_{m+1}=x_1$. To justify the last inequality, expand the expectation
over tuples of original parts. The kernel $\widetilde g$ is constant on
each block. A tuple involving a part contained in $S$ contributes zero;
for every other tuple, each restricted vertex weight is at least
$1/(m+1)$ times its original part weight, by the partwise bound above.

The sets $V^{(1)},\ldots,V^{(m)}$ are disjoint, so every nonzero term on
the left of \eqref{eq:count-compare} represents an injective $C_m$ in a
subgraph of $G$. There are at most $2m$ embeddings per unlabelled cycle.
Hypothesis (iii) therefore bounds that side by $2m\delta$. Consequently,
\[
 t(C_m,\widetilde g)
 \le (m+1)^m(2m+9kDL^m)\delta
 =O_{k,L,C}(\delta).
\]

Apply Theorem~\ref{thm:weighted-removal} to the kernel
$\widetilde g/L\in[0,1]$, on the finite quotient probability space whose
points are the parts of $\mathcal P$, with weights $|V_i|/n$, and use
deletion tolerance $\epsilon/(3L)$. Choose $\delta$ small enough that its
density hypothesis holds. Pulling the resulting symmetric set back to
$V\times V$ gives a symmetric union $A$ of whole blocks with
\[
 \|\widetilde g-\widetilde g\1_A\|_1\le\epsilon/3,
 \qquad t(C_m,\widetilde g\1_A)=0.
\]
If an edge of $\widetilde f$ is positive, its block average is positive.
Since $A$ is a union of blocks, a $C_m$-homomorphism in
$\widetilde f\1_A$ would give one in $\widetilde g\1_A$. Hence
$t(C_m,\widetilde f\1_A)=0$. Also, block averaging gives
\[
 \|\widetilde f-\widetilde f\1_A\|_1
 =\|\widetilde g-\widetilde g\1_A\|_1\le\epsilon/3.
\]
The final removal therefore deletes at most $\epsilon pn^2/6$ undirected
edges. After fixing $\delta$ and $M$, increase $n_0$ so that the small-part
deletion estimate holds. In total we delete at most $\epsilon pn^2/4$
edges between dense pairs, at most $\epsilon pn^2/3$ incident to small
parts, and at most $\epsilon pn^2/3$ in the final step.
The remaining graph is $C_{2k+1}$-homomorphism-free, equivalently
$\{C_3,C_5,\ldots,C_{2k+1}\}$-free: every shorter odd cycle supports a
closed walk of length $2k+1$ by adding backtracking steps, and every odd
closed walk contains an odd cycle of no greater length. This proves the
conclusion.
\end{proof}

We restate Theorem~\ref{thm:C4-removal} in the following quantified form: let $p\ge n^{-1/2}$.
For every $\epsilon>0$ and $k\ge2$, there exist $n_0$ and $\delta>0$ such that, for every
$n\ge n_0$, every $n$-vertex graph with at most $\delta p^4n^4$ copies of $C_4$ and at most
$\delta p^{2k+1}n^{2k+1}$ copies of $C_{2k+1}$ can be made
$\{C_3,C_5,\ldots,C_{2k+1}\}$-free by deleting at most $\epsilon pn^2$ edges.

\begin{proof}[Proof of Theorem~\ref{thm:C4-removal}]
Let $K$ be an absolute constant such that the estimate in
Lemma~\ref{lem:dense-pair} is at most
$K(q^{-1}+M^2+M^{1/4}T^{1/4})n$. First choose $L\ge12K/\epsilon$.
When $\delta\le1$ and $p\ge n^{-1/2}$, the estimate in the proof of
Lemma~\ref{lem:hom-C4} gives
\[
 \homn(C_4,G)\le24N(C_4,G)+3n^2\le27p^4n^4.
\]
Consequently, by Lemma~\ref{lem:galvin-tetali},
\[
 \homn(C_{2k},G)\le\homn(C_4,G)^{k/2}
 \le27^{k/2}p^{2k}n^{2k}\le(3pn)^{2k}.
\]
Apply Theorem~\ref{thm:pascadi} with $\epsilon,L,C=3,k$ to obtain
$M,\delta_*,n_*$. Now choose
\[
 0<\delta\le\min\left\{1,\delta_*,
      \left(\frac{\epsilon}{12KM^{1/4}}\right)^4\right\},
 \qquad
 n_0\ge\max\left\{n_*+1,
      \left(\frac{12KM^2}{\epsilon}\right)^2\right\}.
\]
Let $G$ satisfy the hypotheses in the quantified statement above.
For every partition into at most $M$ parts, Lemma~\ref{lem:dense-pair}
with $q=Lp$ bounds its dense-pair edges by
\[
 K\left(\frac{n}{Lp}+M^2n+M^{1/4}\delta^{1/4}pn^2\right)
 \le\epsilon pn^2/4.
\]
The three terms are each at most $\epsilon pn^2/12$, using $p^2n\ge1$,
$pn\ge\sqrt n$, and the choices of $L,n_0,\delta$, respectively.
Thus conditions (i)--(iii) of Theorem~\ref{thm:pascadi} hold, and its
conclusion proves the result.
\end{proof}

\end{document}